\documentclass[10pt]{article}
\usepackage[a4paper, left=3cm, right=3cm, top=38mm, bottom=38mm]{geometry}
\usepackage{setspace}
\usepackage{mathtools,amssymb,amsthm}
\RequirePackage[dvips]{graphicx}
\usepackage{xcolor}
\usepackage{nicematrix}
\usepackage{threeparttablex}
\usepackage{booktabs}
\usepackage{lipsum}
\usepackage{tikz}

\theoremstyle{plain}
\newtheorem{theorem}{Theorem}[section]
\newtheorem{lemma}[theorem]{Lemma}
\newtheorem{corollary}[theorem]{Corollary}

\newcommand{\bbGamma}{{\mathpalette\makebbGamma\relax}}
\newcommand{\makebbGamma}[2]{%
    \raisebox{\depth}{\scalebox{1}[-1]{$\mathsurround=0pt#1\mathbb{L}$}}}%
\newcommand\bbGammaT{\reflectbox{\rotatebox[origin=c]{180}{$\mathbb L$}}}

\begin{document}
\baselineskip=0.21in

\begin{center}
    \begin{spacing}{1.7}
        {\LARGE \textbf{Kirchhoff index of a nested geometric graph\\ with weighted multiple edges}}

        \vspace{5mm}
        {\large \textbf{Da-yeon Huh}}
    \end{spacing}
        \large \textit{Department of Mathematics, Sungkyunkwan University,}\\
        \large \textit{Suwon 16419, Republic of Korea}\\
        \vspace{3mm}
        {E-mail: \texttt{dayeonhuh7@gmail.com}}
\end{center}

\vspace{3mm}

\begin{abstract}
Kirchhoff index, $K\!f(\mathbb{G})$, introduced by Klein and Randić in 1993, represents the total effective resistances between all pairs of vertices in a graph $\mathbb{G}$, where each edge is regarded as a resistor.
In this paper, the Kirchhoff indices of a particular sequence of nested geometric graphs with weighted multiple edges, denoted by $\mathbb{G}_n$, are investigated.
A recurrence relation for the characteristic polynomial of the Laplacian matrix  $L(\mathbb{G}_n)$ is derived, and an explicit formula for $K\!f(\mathbb{G}_n)$ is obtained.
These facilitate the analysis of the variation of $K\!f(\mathbb{G}_n)$ as
as $n \rightarrow \infty$. 
Consequently, $K\!f(\mathbb{G}_n)$ is shown to grow asymptotically linearly, characterized by a specific asymptotic formula.
In the course of this derivation, a recurrence relation for the determinant of a block tridiagonal matrix is established.
The Kirchhoff index of a 4-regular graph constructed from $\mathbb{G}_n$ is also determined.
\end{abstract}

\vspace{1mm}

    \begin{center}
    \parbox{13cm}{
        }
    \end{center}

\vspace{7mm}

\section{Introduction}
    The Kirchhoff index stems from Klein and Randić's idea of resistance distance \cite{klein1993resistance}.
    The concept of resistance distance between two vertices $u$ and $v$ in a weighted graph $\mathbb{G}$ involves treating $\mathbb{G}$ as an electric circuit, where each edge has a resistance equal to its weight, and measuring the effective (or equivalent) resistance across two vertices $u$ and $v$, denoted $\Omega(u,\,v)$.
    Subsequently, the Kirchhoff index is defined in \cite{bonchev1994molecular} as the sum of resistance distances between every pair of vertices as follows:
                        \[
                        K\!f(\mathbb{G})=\sum_{\{u,v\}\subset V(\mathbb{G})} \Omega(u, v),
                        \]
    where $V(\mathbb{G})$ is the vertex set of $\mathbb{G}$.
    For a detailed analysis of electrical networks via graph theory and matrix algebra, we refer the reader to the renowned classic text by Seshu and Reed \cite{seshu1961linear}.

    We now consider the following simple electrical circuit to briefly summarize the application to weighted graphs with multiple edges presented in this article.
                        \begin{figure}[h]
                        \centering
                        \includegraphics[width=5cm]{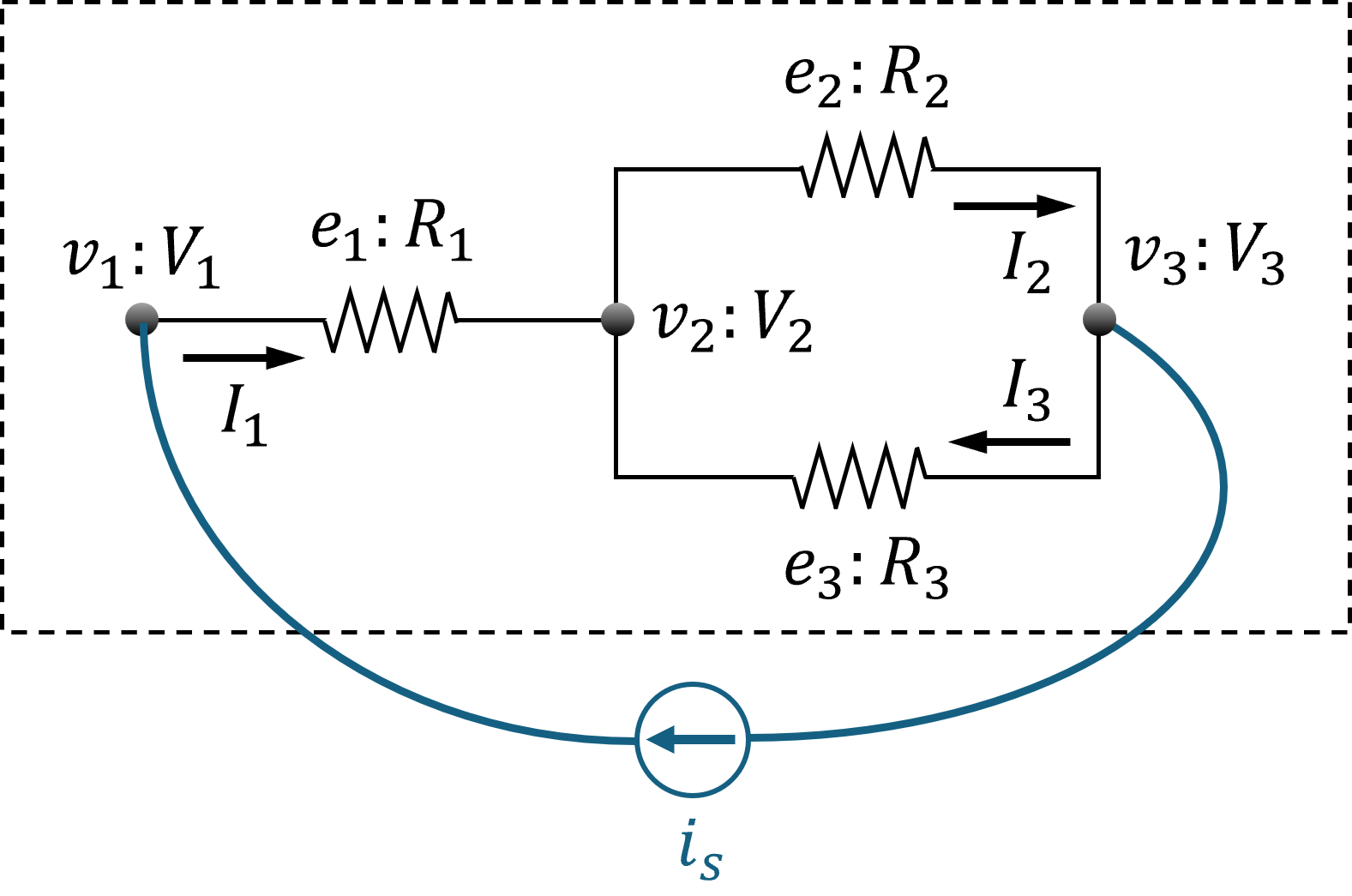}
                         \caption{\small Simple circuit with three resistors} \label{circuit}
                        \end{figure}
    The graph $\mathbb{G}$ within the dotted box in Figure \ref{circuit} consists of three vertices $\{v_{1},v_{2},v_{3}\}$ and three edges $\{e_{1},e_{2},e_{3}\}$, whose weights correspond to resistances $R_{1}$, $R_{2}$, and $R_{3}$, respectively.
    Two of them form parallel (multiple) edges connecting the same pair of vertices $(v_{2},v_{3})$.
    The currents flowing through each edge in the direction of the arrows are represented by $I_1$, $I_2$, and $I_3$ (where $I$ does not denote identity matrix here).
    Assume a test current source of value $i_s$ is connected between $v_{1}$ and $v_{3}$ to determine the value $\Omega (v_{1},v_{3})$.
    Applying Kirchhoff’s current law to each vertex yields
                    \begin{align}\label{QI}
                                    &\nonumber
                                    \begingroup
                                    \renewcommand*{\arraystretch}{0.5}
                                    \begin{matrix}\quad~\, \scriptstyle{e_{1}}~&\scriptstyle{e_{2}}~&\scriptstyle{e_{3}}
                                    \end{matrix}
                                    \endgroup \\
                            \begin{matrix}
                                    \scriptstyle{v_1} \\\scriptstyle{v_2} \\\scriptstyle{v_3}
                            \end{matrix}
                            &\begin{bmatrix*}[r]
                                    1 & 0 & 0 \,\\ -1 & 1 & -1\, \\ 0 & -1 & 1\,
                            \end{bmatrix*}
                            \begin{bmatrix}
                                    \,I_1\, \\ I_2 \\ I_3
                            \end{bmatrix}
                            =\begin{bmatrix*}[r]
                                    i_s \\ 0~ \\ -i_s
                            \end{bmatrix*}.
                    \end{align}
    Since Kirchhoff's current law contains the information about the connectivity between edges and vertices, the $3\times3$ matrix on the left side of equation (\ref{QI}) can be transformed into the form of an incidence matrix $Q$ by sign adjustment, as shown in (\ref{QI}).
    Therefore, if we define the current vector as $\boldsymbol{I}:=[\,I_1, I_2, I_3\,]^T$,
                    \begin{equation}\label{QI2}
                            Q\,\boldsymbol{I}=i_s\,(\boldsymbol{u}_1-\boldsymbol{u}_3)
                    \end{equation}
    is satisfied, where $\boldsymbol{u}_i\in\mathbf{R}^3$ is a unit vector with a $1$ in its $i$-th entry.
    Furthermore, if the potential vector at each vertex is defined as $\boldsymbol{V}:= [V_1, V_2, V_3]^T$,
                    \begin{align*}
                            Q^T\,\boldsymbol{V}=
                            \begin{bmatrix*}[c]
                                    \,V_1-V_2 \,\\ V_2-V_3 \\ V_3-V_2
                            \end{bmatrix*}
                            \,=\,
                            \begin{bmatrix*}[c]
                                    \,R_1 & 0 & 0 \,\\ 0 & \,R_2 & 0 \\ 0 & 0 & \,R_3
                            \end{bmatrix*}
                            \begin{bmatrix*}[c]
                                    \,I_1\, \\ I_2 \\ I_3
                            \end{bmatrix*}
                    \end{align*}
    is also obtained from Ohm's law.
    Assuming $R$ is a diagonal matrix with diagonal entries $R_1$, $R_2$, and $R_3$,
                    \[  R^{-1}Q^T\,\boldsymbol{V}=\boldsymbol{I} \]
    can be derived.
    By multiplying both sides by $Q$ and applying equation (\ref{QI2}), we have
                    \[  Q\,R^{-1}Q^T\,\boldsymbol{V}=i_s\,(\boldsymbol{u}_1-\boldsymbol{u}_3). \]
    If we define the Laplacian matrix $L=L(\mathbb{G})$ of the graph $\mathbb{G}$ with weighted multiple edges as
                    \begin{align*}&L=Q\,R^{-1}Q^T=
                            \begingroup
                            \renewcommand*{\arraystretch}{1.5}
                            \begin{small}\begin{bmatrix}
                                    \frac{1}{R_1}\,&-\frac{1}{R_1}&0\\
                                    -\frac{1}{R_1}\,&\,\frac{1}{R_1}+\frac{1}{R_2}+\frac{1}{R_3}\,&\,-\frac{1}{R_2}-\frac{1}{R_3}\,\\
                                    0&-\frac{1}{R_2}-\frac{1}{R_3}&\frac{1}{R_2}+\frac{1}{R_3}
                            \end{bmatrix}\end{small},
                            \endgroup
                    \end{align*}
    we arrive at the following general expression
                    \begin{align*}\label{LV}
                            L\,\boldsymbol{V}=i_s(\boldsymbol{u}_i-\boldsymbol{u}_j).
                    \end{align*}
    In 1993, Klein et al. \cite{klein1993resistance} demonstrated that the solution for $\Omega(v_i,\,v_j)$ can be expressed as
                    \begin{equation}\label{Resist}
                            \Omega(v_i,\,v_j)=\frac{|V_i-V_j|}{i_s}=(\boldsymbol{u}_i-\boldsymbol{u}_j)^T\,L^{\dagger}(\boldsymbol{u}_i-\boldsymbol{u}_j),
                    \end{equation}
    where $L^{\dagger}$ denotes the Moore-Penrose inverse of $L$, the special case of generalized inverses.  
    The reader is referred to \cite{bapat2004resistance, bapat2010graphs, kelathaya2023generalized} for more information on equation (\ref{Resist}).
    Moreover, it was shown by Klein et al. and Gutman et al. \cite{klein1993resistance, gutman1996quasi} that the Kirchhoff index $K\!f(\mathbb{G})$ satisfies
                    \begin{equation}\label{Kf_eigen}
                            K\!f(\mathbb{G})=n\,\sum_{i=2}^{n}\frac{1}{\lambda_i},
                    \end{equation}
    where $\lambda_i$ are the nonzero eigenvalues of $L(\mathbb{G})$.
    In 2003, Bapat \cite{bapat2003simple} derived another formula for  $\Omega(v_i,\,v_j)$ as given in
                    \begin{equation}\label{Resist_det}
                            \Omega(v_i,\,v_j)=\frac{det L(i,j|i,j)}{det L(i|i)},
                    \end{equation}
    where $L(i|j)$ is the matrix obtained from $L$ by deleting the $i$-th row and $j$-th column, and
    $L(i,j|i,j)$ is defined analogously (see Table \ref{notation}).
    A distinct proof for the formula (\ref{Resist_det}) is given in \cite{kagan2015equivalent}.
    In 2007, $K\!f(\mathbb{G})$ was expressed in terms of two coefficients of the characteristic polynomial $det(\lambda I- L(\mathbb{G}))=\lambda^n+a_{n-1}\lambda^{n-1}+ \cdots + a_1\lambda$, as shown in
                    \begin{align}\label{Kf_poly}
                            K\!f(\mathbb{G})=n\left|\frac{a_2}{a_1}\right|
                    \end{align}
    \cite{gao2012kirchhoff, wang2013laplacian, yang2008kirchhoff}.

    The Kirchhoff index has been extensively studied and applied to diverse network structures.
    Among the numerous contributions in the literature, we briefly mention several works relevant to the present study, including
    silicate networks by Sadar et al. \cite{sardar2020computation},
    linear hexagonal chains by Yang et al. \cite{yang2008kirchhoff},
    polycyclic chains by Liu et al. \cite{liu2024extremal},
    generalized phenylenes by Zhu et al. \cite{zhu2019normalized},
    and $K_n$-chain network by Sun et al. \cite{sun2024resistance}.
    In this article, we introduce and investigate a sequence of nested geometric graphs with weighted multiple edges, as illustrated in Figures \ref{Fig_nested_eye} and \ref{Fig_Nested_withR}, which have not been thoroughly explored.
    To study the variation of the Kirchhoff index of the sequence, we let $\mathbb{G}_{n}$ denote the graph at stage $n$ with order $2n$.
                    \begin{figure}[h]
                        \centering
                        \includegraphics[width=11cm]{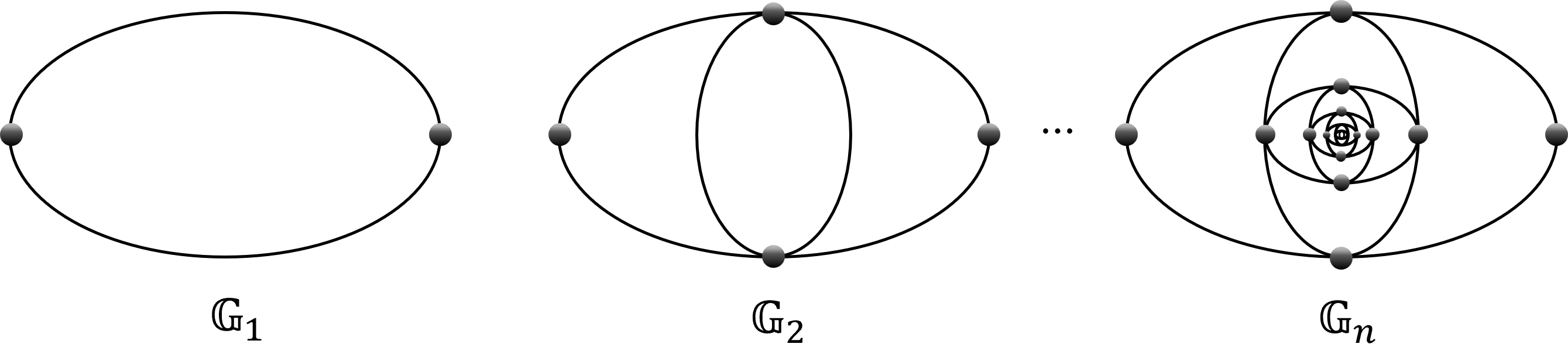}
                         \caption{\small Sequence of nested geometric graphs} \label{Fig_nested_eye}
                    \end{figure}
    In Section 2, we obtain the recurrence formula for the determinant of the block tridiagonal matrix used in computing $K\!f(\mathbb{G}_n)$.
    Section 3 derives the recurrence relation for the characteristic polynomials of $\mathbb{G}_{n}$ and obtains the exact expression of $K\!f(\mathbb{G}_n)$, from which the asymptotic behavior of $K\!f(\mathbb{G}_n)$ is determined.
    Section 4 is devoted to computing the Kirchhoff index of the 4-regular graph generated from $\mathbb{G}_n$.
    While most terms are defined in the text, those used throughout the remainder of this paper are summarized in Table \ref{notation}.
                        \begin{table}[]
                                \centering
                                \begingroup
                                \renewcommand*{\arraystretch}{1.2}
                                \caption{List of Notations} \label{notation}
                                \footnotesize
                                \setlength\tabcolsep{4pt}
                                    \begin{tabular}{l p{6.5cm}}
                                    \toprule
                                    Symbol & Description\\
                                    \midrule
                                    $\mathbf{R}^n$ & $n$-dimensional real column vectors\\
                                    $\mathbf{R}^{m\times n}$ & $m$ by $n$ real matrices\\
                                    $\mathbf{0}_n$ & vector of all zeros $\in \mathbf{R}^{n}$\\
                                    $\mathbf{1}_n$ & vector of all ones $\in \mathbf{R}^{n}$\\
                                    $J_n$ & matrix of all ones $\in \mathbf{R}^{n\times n}$\\
                                    $J_{m,n}$ & matrix of all ones $\in \mathbf{R}^{m\times n}$\\
                                    $O_n$ & zero matrix $\in \mathbf{R}^{n\times n}$\\
                                    $I_n$ & identity matrix $\in \mathbf{R}^{n\times n}$\\
                                    $A_{(i,j)}$ & $(i, j)$ entry of $A \in \mathbf{R}^{m\times n}$\\
                                    $\psi_{A}(\lambda)$ & characteristic polynomial of $A \in \mathbf{R}^{n\times n}$\\
                                    $det\,A$ & determinant of $A \in \mathbf{R}^{n\times n}$\\
                                    $adj\,A$ & transposed cofactor matrix of $A \in \mathbf{R}^{n\times n}$\\
                                    $A(i|j)$ & matrix obtained from $A\in \mathbf{R}^{m\times n}$ by deleting its $i$-th row and $j$-th column\\
                                    $A(i,j|k,l)$ & matrix obtained from $A\in \mathbf{R}^{m\times n}$ by deleting its $i$, $j$-th rows and $k$, $l$-th columns\\
                                    Leading principal submatrix of $A$ & square submatrix at the left top corner of $A\in \mathbf{R}^{n\times n}$\\
                                    \bottomrule
                                    \end{tabular}
                                \endgroup
                        \end{table}

\section{Determinant of a block tridiagonal matrix}
    \begin{lemma}\cite{horn2012matrix} \label{blockdet}
    Suppose $A$ is a nonsingular square submatrix of a block matrix $M=\begin{bmatrix}\,A & B\, \\ \,C & D\,\end{bmatrix}\in \mathbf{R}^{n\times n}$.
    Then
                \begin{equation}\label{eqdet1}
                        det\,M=det\,A~det\,(D- CA^{-1}B).
                \end{equation}
    In particular, when $D=d\in \mathbf{R}$, $B=\boldsymbol{x}$ and $C=\boldsymbol{y}^T$, where $\boldsymbol{x}, \boldsymbol{y}\in \mathbf{R}^{n-1}$,
                \begin{equation}\label{eqdet2}
                        det\,\begin{bmatrix}
                        \,A & \boldsymbol{x}\, \\
                        \,\boldsymbol{y}^T & d\,
                        \end{bmatrix}=d~det\,A-\boldsymbol{y}^T (adj A)~\boldsymbol{x}.
                \end{equation}
    Moreover,
                \begin{equation}\label{eqdet3}
                    det\,(A+\boldsymbol{x}\boldsymbol{y}^T)=det\,A+\boldsymbol{y}^T\,(adj\, A)~\boldsymbol{x}.
                \end{equation}
    \end{lemma}
    \begin{lemma} \label{blockdet2}
    Let $J_{mn}$ be the $m\times n$ matrix of all ones.
    Suppose $A\in \mathbf{R}^{m\times m}$ and $B\in \mathbf{R}^{n\times n}$ are  nonsingular submatrices of a block matrix $M=\begin{bmatrix}\,A & p\,J_{m,n}\, \\ \,q\,J_{n,m} & B\,\end{bmatrix}\in \mathbf{R}^{(m+n)\times (m+n)}$.
    Then
                \[
                det\,M=det\,A~det\,B-p\,q\sum_{\forall(i,j)}(adj\,A)_{(i,j)}\sum_{\forall(i,j)}(adj\, B)_{(i,j)}.
                \]
    \end{lemma}
    \begin{proof}
                By Lemma \ref{blockdet}(\ref{eqdet1}),
                \begin{align*}
                        det\,M=det\,A~det\,(B-p\,q\,J_{n,m}\,A^{-1}\,J_{m,n}).
                \end{align*}
                Let $\boldsymbol{1}_n=[1, \ldots ,1]^T \in \mathbf{R^{n}}$.
                Since
                \begin{align*}
                J_{n,m}\,A^{-1}\,J_{m,n}
                        &=\boldsymbol{1}_n \,\boldsymbol{1}_m^T\, \frac{adj\,A}{det\,A}\, \boldsymbol{1}_m\,\boldsymbol{1}_n^T\\
                        &=\frac{1}{det\,A}\sum_{\forall(i,j)}(adj\,A)_{(i,j)}\,\boldsymbol{1}_n\,\boldsymbol{1}_n^T,
                \end{align*}
    by Lemma \ref{blockdet}(\ref{eqdet3}) we have
                \begin{align*}
                        det\,M
                        &=det\,A\,det\bigg(B-\frac{p\,q}{det\,A}\sum_{\forall(i,j)}(adj\,A)_{(i,j)}\,\boldsymbol{1}_n\,\boldsymbol{1}_n^T\bigg)\\
                        &=det\,A\,\bigg(det\,B-\frac{p\,q}{det\,A}\sum_{\forall(i,j)}(adj\,A)_{(i,j)}\,\boldsymbol{1}_n^T\,(adj\,B)\,\boldsymbol{1}_n\bigg)\\
                        &=det\,A\,det\,B-p\,q\sum_{\forall(i,j)}(adj\,A)_{(i,j)}\sum_{\forall(i,j)}(adj\,B)_{(i,j)}.
                \end{align*}
    \end{proof}

    We define a block tridiagonal matrix $T_n\in \mathbf{R}^{2n\times 2n}$ by
                \begingroup
                        \renewcommand*{\arraystretch}{1.5}
                        \[
                        T_n:=\begin{bmatrix}
                        \,A_1 & p_1J_2& O_2 & \cdots & O_2\,\\
                        \,q_1J_2 & A_2 & p_2J_2 & & \vdots \\
                        \,O_2& q_2J_2 & A_3 & \ddots & O_2\, \\
                        \vdots& & \ddots & \ddots & p_{n-1} J_2\,\\
                        \,O_2 &\cdots &O_2 & q_{n-1}J_2 & A_n\,
                        \end{bmatrix},
                        \]
                \endgroup
    where
                $A_i=\begin{bmatrix}  a_i&b_i\\  c_i&d_i    \end{bmatrix}$,
                $J_2=\begin{bmatrix}     \,1&1\,\\   \,1&1\,    \end{bmatrix}$,
                $O_2=\begin{bmatrix}     \,0&0\,\\    \,0&0\,    \end{bmatrix}$, and $a_i, b_i, c_i, d_i, p_i, q_i \in \mathbf{R}$.
    \begin{theorem}\label{thm1}
    Suppose that each $A_i$ $(i=1, ... , n)$ in $T_n$ is nonsingular.
    Let $\boldsymbol{g_n}=det\,T_n$.
    Then the sequence $\boldsymbol{g_n}$ satisfies the following recurrence relation:
                \[
                        \boldsymbol{g_n}=det\,A_n~\boldsymbol{g_{n-1}}
                        -p_{n-1}\,q_{n-1}\,t_n\,t_{n-1}\,\boldsymbol{g_{n-2}} \qquad (n\geq 3),\\
                \]
    where $t_i=a_i+d_i-b_i-c_i$, with the initial values
                \begin{align*}
                        &\boldsymbol{g_1}=a_1d_1-b_1c_1,\\
                        &\boldsymbol{g_2}=det\,A_1~det\,A_2-p_1\,q_1\,t_1\,t_2.
                \end{align*}
    \end{theorem}
    \begin{proof}
                $\boldsymbol{g_1}=det\,T_1=det\,A_1=a_1d_1-b_1c_1$ is trivial.
                $\boldsymbol{g_2}$ can be obtained by Lemma \ref{blockdet2} as
                \begin{align*}
                \boldsymbol{g_2}&=det\,T_2
                        =\begin{bmatrix} \,A_1&p_1J_2\,\\ \,q_1 J_2&A_2\, \end{bmatrix}\\
                        &=det\,A_1\,det\,A_2-p_1\,q_1\,(a_1+d_1-b_1-c_1)\,(a_2+d_2-b_2-c_2)\\
                        &=det\,A_1\,det\,A_2-p_1\,q_1\,t_1\,t_2.
                \end{align*}
    Next, applying \ref{blockdet}(\ref{eqdet1}) to $T_n=\begin{bmatrix} \,T_{n-1}&[\,O_2 \cdots ~ p_{n-1}J_2\,]^T \,\\ \,[\,O_2 \cdots ~ q_{n-1}J_2\,]&A_n\, \end{bmatrix}$ gives
                \[
                    \boldsymbol{g_n}=det\,T_n=det\,A_n\,det\,(T_{n-1}~-~p_{n-1}\,q_{n-1}[\,O_2 \cdots ~ J_2\,]^T\,A_n^{-1}\,[\,O_2 \cdots ~ J_2\,]).
                \]
                \vskip 3pt
    \noindent Note that $[\,O_2 \cdots ~ p_{n-1}J_2\,]^T \in \mathbf{R}^{(2n-2)\times 2}$ and $[\,O_2 \cdots ~ q_{n-1}J_2\,] \in \mathbf{R}^{2\times (2n-2)}$.
    Let $\boldsymbol{0}_n=[0, ... ,0]^T \in \mathbf{R^{n}}$.
    Since
                \begin{align*}
                        &[\,O_2 \cdots ~ J_2\,]^T\,A_n^{-1}\,[\,O_2 \cdots ~ J_2\,]\\
                        &=\begin{bmatrix}\,O_2 & \cdots &O_2\,\\ \vdots & \ddots & \vdots\\\,O_2 & \cdots & J_2\,A_n^{-1}\,J_2  \,\end{bmatrix}
                        =\frac{(a_{n}+d_{n}-b_{n}-c_{n})}{det\,A_n}
                        \begingroup
                        \renewcommand*{\arraystretch}{1.2}
                            \begin{bmatrix}\,\boldsymbol{0}_{2n-4}\,\\~\boldsymbol{1}_2\,\end{bmatrix}
                        \endgroup
                        [\,\boldsymbol{0}_{2n-4}^{\,T}~~\boldsymbol{1}_2^T\,],
                \end{align*}
    we have by Lemma \ref{blockdet}(\ref{eqdet3}) that
                \begin{align*}
                    \boldsymbol{g_n}&=det\,A_n\,det\,\left(T_{n-1}-p_{n-1}\,q_{n-1}
                    \frac{t_n}{det\,A_n}
                        \begingroup
                        \renewcommand*{\arraystretch}{1.2}
                            \begin{bmatrix}\,\boldsymbol{0}_{2n-4}\,\\~\boldsymbol{1}_2\,\end{bmatrix}
                        \endgroup
                        [\,\boldsymbol{0}_{2n-4}^{\,T}~~\boldsymbol{1}_2^T\,]\right)\\
                    &=det\,A_n\,det\,T_{n-1}-p_{n-1}\,q_{n-1}\,t_n
                        [\,\boldsymbol{0}_{2n-4}^{\,T}~~\boldsymbol{1}_2^T\,](adj\,T_{n-1})
                        \begingroup
                        \renewcommand*{\arraystretch}{1.2}
                            \begin{bmatrix}\,\boldsymbol{0}_{2n-4}\,\\~\boldsymbol{1}_2\,\end{bmatrix}
                        \endgroup.
                \end{align*}
    Let $\boldsymbol{p}=[\boldsymbol{0}_{2n-6},p_{n-2},p_{n-2}]^T \in \mathbf{R^{2n-4}}$ and $\boldsymbol{q}=[\boldsymbol{0}_{2n-6},q_{n-2},q_{n-2}]^T \in \mathbf{R^{2n-4}}$.
    Then
                \begin{align*}
                        &[\,\boldsymbol{0}_{2n-4}^{\,T}~~\boldsymbol{1}_2^T\,](adj\,T_{n-1})
                        \begingroup
                        \renewcommand*{\arraystretch}{1.2}
                            \begin{bmatrix}\,\boldsymbol{0}_{2n-4}\,\\~\boldsymbol{1}_2\,\end{bmatrix}
                        \endgroup \\
                        &=det\,T_{n-1}(2n-2|2n-2)+det\,T_{n-1}(2n-3|2n-3)\\
                        &\quad -det\,T_{n-1}(2n-2|2n-3)-det\,T_{n-1}(2n-3|2n-2)\\
                        &=det\begin{bmatrix}\,T_{n-2} & \boldsymbol{p} \\ \boldsymbol{q}^T &a_{n-1}\end{bmatrix}
                            +det\begin{bmatrix}\,T_{n-2} & \boldsymbol{p} \\ \boldsymbol{q}^T &d_{n-1}\end{bmatrix}
                            -det\begin{bmatrix}\,T_{n-2} & \boldsymbol{p} \\ \boldsymbol{q}^T &b_{n-1}\end{bmatrix}
                            -det\begin{bmatrix}\,T_{n-2} & \boldsymbol{p} \\ \boldsymbol{q}^T &c_{n-1}\end{bmatrix}\\
                        &=(a_{n-1}det\,T_{n-2}-\boldsymbol{q}^T adj\,T_{n-2}\,\boldsymbol{p})
                            +(d_{n-1}det\,T_{n-2}-\boldsymbol{q}^T adj\,T_{n-2}\,\boldsymbol{p})\\
                        &\quad -(b_{n-1}det\,T_{n-2}-\boldsymbol{q}^T adj\,T_{n-2}\,\boldsymbol{p})
                            -(c_{n-1}det\,T_{n-2}-\boldsymbol{q}^T adj\,T_{n-2}\,\boldsymbol{p})\\
                        &=(a_{n-1}+d_{n-1}-b_{n-1}-c_{n-1})\,det\,T_{n-2}
                \end{align*}
    by applying Lemma \ref{blockdet}(\ref{eqdet2}).
    Finally,
                \begin{align*}
                        \boldsymbol{g_n}&=det\,A_n\,det\,T_{n-1}-p_{n-1}\,q_{n-1}\,t_n
                        (a_{n-1}+d_{n-1}-b_{n-1}-c_{n-1})\,det\,T_{n-2}\\
                        &=det\,A_n\,\boldsymbol{g_{n-1}}-p_{n-1}\,q_{n-1}\,t_n\,
                        t_{n-1}\,\boldsymbol{g_{n-2}}.
                \end{align*}
    \end{proof}

    When the submatrix $A_i$ of $T_n$ is in the form of
    $A_i=\begin{bmatrix}  a_i&b_i\\  b_i&a_i    \end{bmatrix}$,
    the recurrence relation for the determinant $T_n$ can be obtained more simply using the following two lemmas.
    \begin{lemma}\cite{muir2003treatise}\label{lem_tridiagdet}
    Let $\boldsymbol{h_n}$ denote the determinant of the $n \times n$ tridiagonal matrix defined by
                \[
                        \boldsymbol{h_n}:=det\begin{bmatrix}
                        \,a_1 & p_1& 0 & \cdots & 0\,\\
                        \,q_1 & a_2 & p_2 & & \vdots\\
                        0 & q_2 & a_3 & \ddots & 0\\
                        \vdots & & \ddots & \ddots & p_{n-1} \,\\
                        \,0 & \cdots & 0 & q_{n-1} & a_n\,
                        \end{bmatrix}.
                \]
    Then $\boldsymbol{h_n}$ satisfies the following recurrence relation
                \[
                \boldsymbol{h_n}=a_n \boldsymbol{h_{n-1}}-p_{n-1}q_{n-1}\boldsymbol{h_{n-2}} \qquad (n\geq 2)
                \]
    with the initial values $\boldsymbol{h_1}=a_1$ and $\boldsymbol{h_0}=1$.
    \end{lemma}

    For the next lemma and the following text, we denote by
                $\psi_M (\lambda)=det(\lambda I -M)$
    the characteristic polynomial of a square matrix $M$.
    Yang and yu \cite{yang1985graph} considered the unitary transformation $UMU^T$ using a unitary matrix
                $U=\frac{1}{\sqrt{2}}\begin{bmatrix}\,I_n & I_n\, \\ \,I_n & -I_n\,\end{bmatrix}$
    to obtain the following results.
    \begin{lemma} \cite{yang1985graph} \label{lem_yang}
    Suppose $A$ and $B$ are square submarices of a block matrix $M=\begin{bmatrix}\,A & B\, \\ \,B & A\,\end{bmatrix}\in \mathbf{R}^{2n\times 2n}$.
    Then
                \begin{align*}
                        &det (M)=det\,(A+B) \,det\,(A-B)\\
                        &\psi_M (\lambda)=\psi_{A+B}(\lambda) \, \psi_{A-B}(\lambda).
                \end{align*}
    \end{lemma}
    \begin{theorem} \label{thm_sym}
    Suppose each $A_i$ $(i=1, ... , n)$ in Theorem \ref{thm1} is of the form
    $A_i=\begin{bmatrix}  a_i&b_i\\  b_i&a_i    \end{bmatrix}$.
    Then
                \[
                \boldsymbol{g_n}=\boldsymbol{\tilde{h}_n}\,\prod_{i=1}^{n}(a_i-b_i),
                \]
    where $\boldsymbol{\tilde{h}_n}$ satisfies the recurrence relation
                \[
                        \boldsymbol{\tilde{h}_n}=(a_n+b_n)\,\boldsymbol{\tilde{h}_{n-1}}-4\,p_{n-1}\,q_{n-1}\boldsymbol{\tilde{h}_{n-2}} \qquad (n\geq 2),
                \]
    with the initial values $\boldsymbol{\tilde{h}_1}=a_1+b_1$ and $\boldsymbol{\tilde{h}_0}=1$.
    Here, $\boldsymbol{\tilde{h}_n}$ represents
                \[
                        \boldsymbol{\tilde{h}_n}=det
                         \begin{bmatrix}
                                        \,a_1+b_1 & 2\,p_1& 0 & \cdots & 0\,\\
                                        \,2\,q_1 & a_2+b_2 & 2\,p_2 & & \vdots\\
                                        0 & 2\,q_2 & a_3+b_3 & \ddots & 0\\
                                        \vdots & & \ddots & \ddots & 2\,p_{n-1} \,\\
                                        \,0 & \cdots \textcolor{white}{\vdots}& 0 & 2\,q_{n-1} & a_n+b_n\,
                        \end{bmatrix}.
                \]
    \end{theorem}
    \begin{proof}
    Consider a $2n \times 2n$ permutation matrix
                \[ P=
                        \small
                        \left[
                                \begin{array}
                                        {@{\hspace{4pt}}c@{\hspace{8pt}}c
                                          @{\hspace{6pt}}c@{\hspace{6pt}}c
                                          @{\hspace{6pt}}c@{\hspace{6pt}}c
                                          @{\hspace{6pt}}c@{\hspace{6pt}}c
                                          @{\hspace{6pt}}c@{\hspace{8pt}}c}
                                1 & 0 & \cdots &  &  &  &  &  &  & 0\\
                                0 & 0 & 1 & 0 & \cdots &  &  &  &  & 0\\
                                0 & 0 & 0 & 0 & 1 & 0 & \cdots &  &  & 0\\
                                  &  &  &  & \ddots &  &  &  &  & \\
                                0 & \cdots &  & \cdots &  & \cdots &  & 0 & 1 & 0\\
                                0 & 1 & 0 & \cdots &  &  &  &  &  & 0\\
                                0 & 0 & 0 & 1 & 0 & \cdots &  &  & & 0\\
                                0 & 0 & 0 & 0 & 0 & 1 & 0 & \cdots &  & 0\\
                                 &  &  &  & \ddots &  &  &  &  & \\
                                0 & \cdots &  & \cdots &  & \cdots &  & \cdots & 0 & 1\\
                                \end{array}
                        \right],
                \]
    which reorders the rows and columns of $T_{n}$ via $P\,T_n\,P^T$ so that the original indices $(1, 2, ... , 2n)$ are rearranged into $(1, 3, 5, ... , 2n-1, 2, 4, 6, ... , 2n)$.
    Then
                \[
                det\,T_n=det\,P\,T_n\,P^T=det\,\begin{bmatrix}\,A & B\, \\ \,B & A\,\end{bmatrix},
                \]
    where
                \[
                A=
                        \begin{bmatrix}
                                        a_1 & p_1& 0 & \cdots & 0\,\\
                                        q_1 & a_2 & p_2 & & \vdots\\
                                        0 & q_2 & a_3 & \ddots & 0\\
                                        \vdots & & \ddots & \ddots & p_{n-1} \,\\
                                        0 & \cdots & 0 & q_{n-1} & a_n\,
                        \end{bmatrix}
                \quad\text{and}\quad
                B=
                        \begin{bmatrix}
                                        b_1 & p_1 & 0 & \cdots & 0\,\\
                                        q_1 & b_2 & p_2 & & \vdots\\
                                        0  & q_2 & b_3 & \ddots & 0\\
                                        \vdots & & \ddots & \ddots & p_{n-1} \,\\
                                        0 & \cdots &  & q_{n-1} & b_n\,
                        \end{bmatrix}.
                \]
    Therefore, Lemma $\ref{lem_yang}$ yields
                \[
                \boldsymbol{g_n}=det\,T_n=det(A+B)\,det(A-B)=\boldsymbol{\tilde{h}_n}\,\prod_{i=1}^{n}(a_i-b_i).
                \]
    Moreover, Lemma $\ref{lem_tridiagdet}$ ensures that $\boldsymbol{\tilde{h}_n}$ satisfies the recurrence relation stated in Theorem $\ref{thm_sym}$.
    \end{proof}

\section{Kirchhoff index of $\mathbb{G}_n$}
    The sequence of graphs $\mathbb{G}_{n}$ is defined recursively, starting from $\mathbb{G}_{1}$ as illustrated in Figure \ref{Fig_Nested_withR}.
    The recursive rule for generating $\mathbb{G}_{n+1}$ from $\mathbb{G}_{n}$ is given as follows:
    \begin{quote}
        \textbf{Step 1)} Two vertices $v_{n+1}$ and $v_{n+1}'$ are added to $\mathbb{G}_{n}$, placing one vertex at the midpoint of each of the two innermost edges.\\
        \textbf{Step 2)} Two multiple edges $(v_{n+1}, v_{n+1}')$ are then connected to these new vertices.\\
        \textbf{Step 3)} The resistance (weight) of the newly created edge is half the resistance of the edge it is connected to.\\
        \textbf{Step 4)} The innermost edge of $\mathbb{G}_{n}$ is divided into two edges due to the added vertex, and each of them is assigned half of the original resistance.
    \end{quote}
                    \begin{figure}[h]
                        \centering
                        \includegraphics[width=12cm]{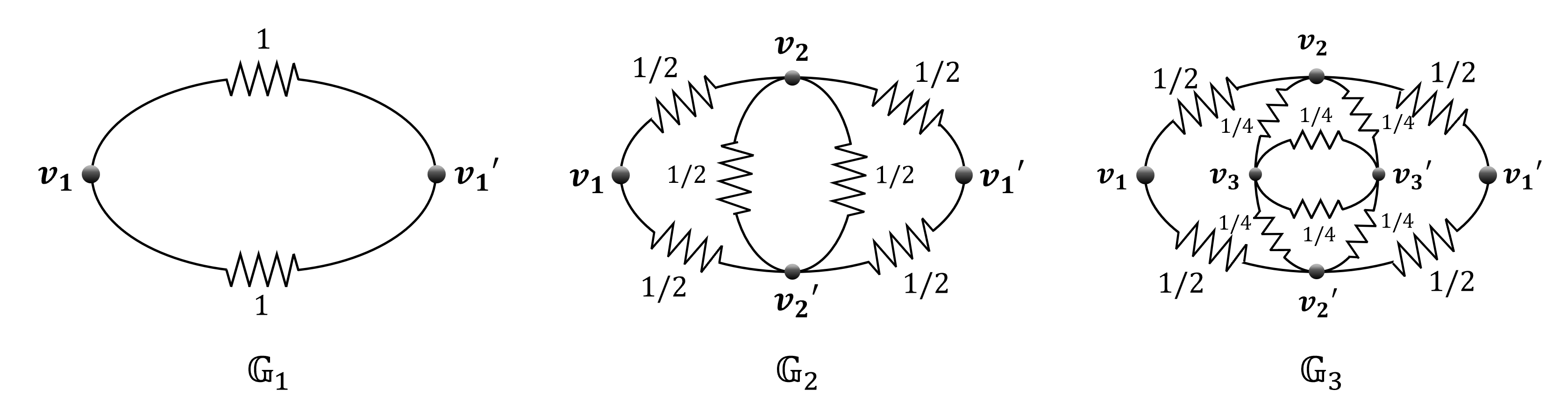}
                        \caption{\small Sequence of nested goemetric graphs with resistance}\label{Fig_Nested_withR}
                    \end{figure}
    \begin{theorem} \label{thm_polyofnest}
    The sequence $\psi_{L(\mathbb{G}_n)}(\lambda)$ satisfies the following equation
                \[
                        \psi_{L(\mathbb{G}_n)}(\lambda)=\left((\lambda-2^n)g_{n-1}(\lambda)
                        -4^{n}g_{n-2}(\lambda)\right)(\lambda-4)\prod_{i=2}^{n}(\lambda-3\cdot2^i)
                        \quad (n\geq 2)
                \]
    with the initial condition
                \begin{align*}
                        \psi_{L(\mathbb{G}_1)}(\lambda)=\lambda\,(\lambda-4),
                \end{align*}
    where $g_{n}(\lambda)$ meets the recurrence relation
                \[
                        g_{n}(\lambda)=(\lambda-3\cdot2^n)\,g_{n-1}(\lambda)
                        -2^{2n}\,g_{n-2}(\lambda)
                        \qquad (n\geq 2),
                \]
    with the initial conditions
                \[
                        g_{1}(\lambda)=\lambda-4 \text{~~and~~} g_{0}(\lambda)=1.
                \]
    \end{theorem}
    \begin{proof}
    $\mathbb{G}_1$ has two multiple edges or two unit resistors in parallel.
    Hence,
                \[
                \psi_{L(\mathbb{G}_1)}(\lambda)
                        =det\,\left[\lambda I_2-\left[\begin{array}{rr}2 & -2\\ -2 & 2\end{array}\right]\right]
                        =\lambda(\lambda-4).
                \]
    $\mathbb{G}_2$ consists of six $r_2=1/2 \,\Omega$ resistors, including two parallel resistors
    as shown in Figure \ref{Fig_Nested_withR}.
    Arbitrary directions are assigned to each edge of $\mathbb{G}_2$, as in Figure \ref{Fig_Nested_DirG2}, to
    induce the incidence matrix $Q$.
                \begin{figure}[h]
                        \centering
                        \includegraphics[width=4cm]{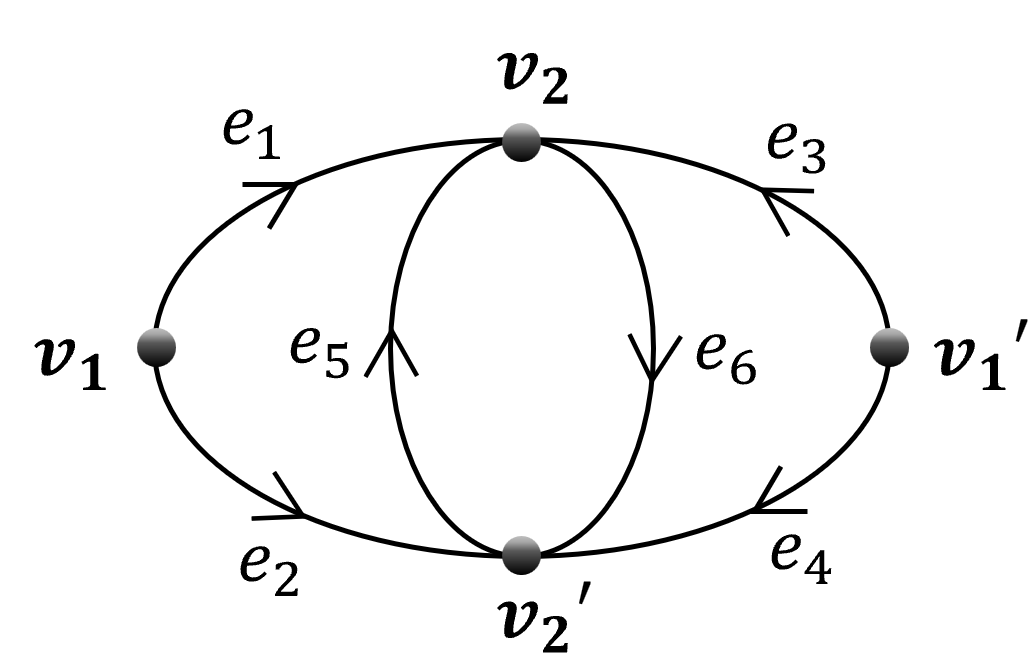}
                        \caption{\small Directed graph of $\mathbb{G}_2$ used to derive $Q$}\label{Fig_Nested_DirG2}
                \end{figure}
                \begin{align*}
                        &\begingroup
                                \renewcommand*{\arraystretch}{0}
                                \begin{matrix} ~~ &\scriptstyle{e_{1}}&\scriptstyle{~e_{2}\,}&\scriptstyle{~e_{3}\,}&\scriptstyle{~e_{4}\,}&\scriptstyle{~e_{5}\,}&\scriptstyle{~e_{6}}
                                \end{matrix}
                        \endgroup \\
                        Q=\begin{matrix}\scriptstyle{v_{1}}\\\scriptstyle{v_{1}'}
                        \\\scriptstyle{v_{2}}\\\scriptstyle{v_{2}'}\end{matrix}
                        &\begin{bmatrix*}[r]
                                    1&1&0&0&0&0\\0&0&1&1&0&0\\-1&0&-1&0&-1&1\\0&-1&0&-1&1&-1
                        \end{bmatrix*}.
                \end{align*}
    Given $Q$ and resistance matrix $D=r_2I_4$, $L(\mathbb{G}_2)$ is given by
                \begin{align*}
                        &L(\mathbb{G}_2)=QD^{-1}Q^T\\
                        &=\begin{matrix}\scriptstyle{v_{1}}\\\scriptstyle{v_{1}'}
                        \\\scriptstyle{v_{2}}\\\scriptstyle{v_{2}'}\end{matrix}
                        \begin{bmatrix}
                                                \scriptstyle{\frac{2}{r_2}}&&&\\
                                                &\scriptstyle{\frac{2}{r_2}}&&\\
                                                &&\scriptstyle{\frac{4}{r_2}}&\\
                                                &&&\scriptstyle{\frac{4}{r_2}}
                        \end{bmatrix}
                        -\begin{bmatrix}
                                                \scriptstyle{0}&\scriptstyle{0}&\scriptstyle{\frac{1}{r_2}}&\scriptstyle{\frac{1}{r_2}}\\
                                                \scriptstyle{0}&\scriptstyle{0}&\scriptstyle{\frac{1}{r_2}}&\scriptstyle{\frac{1}{r_2}}\\
                                                \scriptstyle{\frac{1}{r_2}}&\scriptstyle{\frac{1}{r_2}}&\scriptstyle{0}&\scriptstyle{\frac{1}{r_2}}+\scriptstyle{\frac{1}{r_2}}\\
                                                \scriptstyle{\frac{1}{r_2}}&\scriptstyle{\frac{1}{r_2}}&\scriptstyle{\frac{1}{r_2}}+\scriptstyle{\frac{1}{r_2}}&\scriptstyle{0}
                        \end{bmatrix}
                        =\left[
                            \begingroup
                                \renewcommand*{\arraystretch}{1.2}
                            \begin{array}{@{}r@{\hspace{8pt}}r@{\hspace{8pt}}r@{\hspace{8pt}}r}
                                                4&0&-2&-2\\0&4&-2&-2\\-2&-2&8&-4\\-2&-2&-4&8\end{array}
                            \endgroup
                        \right].
                \end{align*}
    The general form of $L(\mathbb{G}_{n})\in \mathbf{R}^{2n\times 2n}$ for $n\geq2$ is as follows:
                \begin{align*}
                        &\begin{matrix} ~ &\scriptstyle{v_1~v_1'}&\scriptstyle{v_2~v_2'}
                                        &\scriptstyle{v_3~v_3'}
                                        &\scriptstyle{~~\cdots\cdots\cdots\cdots~}
                                        &\scriptstyle{~v_{n-1}~v_{n-1}'}
                                        &\scriptstyle{~~v_n~v_n'}
                                \end{matrix}
                        \\
                        L(\mathbb{G}_n)=
                        &\small
                                \begingroup
                                \renewcommand*{\arraystretch}{1.6}
                                \begin{bmatrix*}[r]
                                        4I_2 & -2J_2 &O_2&&\cdots&O_2 \\
                                        -2J_2 & 12I_2 & -4J_2 &&&\vdots~~\\
                                        O_2& -4J_2 & 24I_2 & \ddots && \\
                                        && \ddots & \ddots & -2^{n-2}J_2 & O_2\\
                                        \vdots&&& -2^{n-2}J_2 & 3\cdot 2^{n-1}I_2 & -2^{n-1}J_2 \\
                                        O_2&\cdots&\textcolor{white}{\vdots}&O_2& -2^{n-1}J_2 & S~
                                \end{bmatrix*},
                                \endgroup
                \end{align*}
    where $S=\small\left[\begin{array}{cc}2^{n+1} & -2^{n}\\ -2^{n}& 2^{n+1}\end{array}\right]$.
    Since $L(\mathbb{G}_n)$ satisfies the conditions of Theorem \ref{thm_sym},
    applying this theorem alongside Lemma $\ref{lem_yang}$ yields
                \begin{align*}
                        \psi_{L(\mathbb{G}_n)}(\lambda)
                        &=\psi_{PL(\mathbb{G}_n)P^T}(\lambda) \\
                        &=det\left[\lambda\,I_{2n}- \left[\begin{array}{rr}A_n & B_n\\ B_n & A_n\end{array}\right]\right] \nonumber\\
                        &=\psi_{A_n+B_n}(\lambda)\,\psi_{A_n-B_n}(\lambda),
                \end{align*}
    where
                \begin{align*}
                                A_n &=
                                        \setlength{\arraycolsep}{4pt}
                                        \begin{bmatrix*}[r]
                                        4 & -2 & 0 & &\cdots& 0~~ \\
                                        -2 & 12 & -4 &&&\vdots~~\\
                                        0& -4 & 24 & \ddots && \\
                                        & &\ddots & \ddots & -2^{n-2} & 0~~\\
                                        \vdots&&& -2^{n-2} & 3\cdot 2^{n-1} & -2^{n-1} \\
                                        0&\textcolor{white}{\vdots}&\cdots&0& -2^{n-1} & 2^{n+1}
                                        \end{bmatrix*}
                                \quad\text{and}\quad B_n =
                                    \setlength{\arraycolsep}{4pt}
                                        \begin{bmatrix*}[r]
                                        0 & -2 & 0 &&\cdots& 0~~ \\
                                        -2 & 0 & -4 &&&\vdots~~\\
                                        0 & -4 & 0 & \ddots && \\
                                        && \ddots & \ddots & -2^{n-2}&0~~\\
                                        \vdots&&&-2^{n-2}  & 0~~ & -2^{n-1} \\
                                        0&\textcolor{white}{\vdots}&\cdots& 0& -2^{n-1} & -2^n~
                                        \end{bmatrix*}
                \end{align*}
    consequently,
                \begin{equation}\label{eqn_ApB}
                                A_n+B_n =
                                        \setlength{\arraycolsep}{4pt}
                                        \begin{bNiceMatrix}[r]
                                        4 & -4 & 0 &&\cdots& 0~~ \\
                                        -4 & 12 & -8 &&&\vdots~~\\
                                        0& -8 & 24 & \ddots && \\
                                        && \ddots & \ddots & -2^{n-1} & 0~~\\
                                        \vdots&&& -2^{n-1} & 3\cdot 2^{n-1} & -2^{n} \\
                                        0&\cdots&\textcolor{white}{\vdots}&0& -2^{n} & 2^n
                                        \CodeAfter
                                            \tikz \draw[blue, densely dashed] ([xshift=1pt]6-|1) -| (1-|6) ;
                                        \end{bNiceMatrix}
                \end{equation}
    and
                \begin{equation*}
                                A_n-B_n =
                                        \setlength{\arraycolsep}{4pt}
                                        \begin{bmatrix*}[r]
                                        ~4 & 0 &  &&\cdots& 0~~ \\
                                        ~0 & 12 & &&&\vdots~~\\
                                        &  & 24 & &&\textcolor{white}{\vdots} \\
                                        && & \ddots & & \\
                                        ~\vdots&&&  & 3\cdot 2^{n-1} & 0~~\\
                                        ~0&\textcolor{white}{\vdots}&\cdots& & 0& 3\cdot2^n
                                        \end{bmatrix*}
                \end{equation*}
    For $A_n-B_n$, we have
                \[
                        \psi_{A_n-B_n}(\lambda)=(\lambda-4)\prod_{i=2}^{n}(\lambda-3\cdot2^i).
                \]
    To derive the recurrence relation for $\psi_{A_n+B_n}(\lambda)$,
    we define $T_n$ as a sequence of matrices having the structure indicated by the blue dashed box of $A_n+B_n$ in (\ref{eqn_ApB}) as
                \begin{align*}
                T_n&=
                    \begingroup
                    \renewcommand*{\arraystretch}{1.1}
                        \begin{bmatrix*}[r]
                                        4 & -4 & 0 &\cdots& 0~~ \\
                                        -4 & 12 & -8 &&\vdots~~\\
                                        0& -8 & 24 & \ddots &0~~ \\
                                        \vdots&&\ddots& \ddots & -2^{n}\\
                                        0&\cdots\textcolor{white}{\vdots}&0& -2^{n} & 3\cdot 2^{n}
                        \end{bmatrix*}
                    \endgroup  \qquad (n\geq 2),\\
                T_1&=[4].
                \end{align*}
    Then, by Lemma \ref{lem_tridiagdet}, $\psi_{T_n}(\lambda)$ satisfies the recurrence relation
                \begin{align*}
                \psi_{T_n}(\lambda)&=(\lambda-3\cdot2^n)\,\psi_{T_{n-1}}(\lambda)
                                    -2^{2n}\,\psi_{T_{n-2}}(\lambda) \qquad (n\geq 2),\\
                \psi_{T_1}(\lambda)&=\lambda-4 \text{~~and~~} \psi_{T_0}(\lambda)=1.
                \end{align*}
    Hence, we obtain
                \[
                        \psi_{A_n+B_n}(\lambda)=(\lambda-2^n)\psi_{T_{n-1}}(\lambda)
                                                   -2^{2n}\psi_{T_{n-2}}(\lambda) \qquad (n\geq 2)
                \]
    as a consequence of the same Lemma again.
    Finally, if we replace $\psi_{T_n}(\lambda)$ with $g_n(\lambda)$ for simplicity,
                 \begin{align*}
                        \psi_{L(\mathbb{G}_n)}(\lambda)
                        &=\psi_{A_n+B_n}(\lambda)\,\psi_{A_n-B_n}(\lambda)\\
                        &=\left((\lambda-2^n)g_{n-1}(\lambda)
                        -2^{2n}g_{n-2}(\lambda)\right)(\lambda-4)\prod_{i=2}^{n}(\lambda-3\cdot2^i)
                        \quad (n\geq 2)
                \end{align*}
    with the initial condition
                $\psi_{L(\mathbb{G}_1)}(\lambda)=\lambda\,(\lambda-4)$,
    where $g_{n}(\lambda)$ is shown to meet the recurrence relation
                \[
                        g_{n}(\lambda)=(\lambda-3\cdot2^n)\,g_{n-1}(\lambda)
                        -2^{2n}\,g_{n-2}(\lambda)
                        \qquad (n\geq 2),
                \]
    with the initial conditions
                $
                        g_{1}(\lambda)=\lambda-4 \text{~~and~~} g_{0}(\lambda)=1.
                $
    \end{proof}
    The characteristic polynomial of $L(\mathbb{G}_n)$, computed recursively using MATLAB (version R2023a) \cite{Matlab} based on $\ref{thm_polyofnest}$, are as follows:
                \begin{quote}
                        $\scriptstyle{\psi_{L(\mathbb{G}_1)}(\lambda)=\lambda^2-4\,\lambda}$\\
                        $\scriptstyle{\psi_{L(\mathbb{G}_2)}(\lambda)=\lambda^4-24\,\lambda^3+176\,\lambda^2-384\,\lambda}$\\
                        $\scriptstyle{\psi_{L(\mathbb{G}_3)}(\lambda)=\lambda^6-64\,\lambda^5+1488\,\lambda^4-15360\,\lambda^3
                                                            +69120\,\lambda^2-110592\,\lambda}$\\
                        $\scriptstyle{\psi_{L(\mathbb{G}_4)}(\lambda)=\lambda^8-144\,\lambda^7+8016\,\lambda^6-220416\,\lambda^5
                        +3192320\,\lambda^4-24023040\,\lambda^3+85524480\,\lambda^2-113246208\,\lambda}$\\
                        $\scriptstyle{\psi_{L(\mathbb{G}_5)}(\lambda)=\lambda^{10}-304\lambda^9
                        +36688\lambda^8-2281216\lambda^7+79869440\lambda^6
                        -1620463616\lambda^5+2147483647\lambda^4}$\\
                        $\scriptstyle{~~~~\quad\quad-2147483648\lambda^3+2147483647\lambda^2-2147483648\lambda}$\\
                        $~\quad \quad \quad \quad \quad \vdots$
                \end{quote}
     The degree of the characteristic polynomial increases by two for every increment of $n$.
     As $n$ grows, the absolute values of the coefficients increase significantly.
    Using $\psi_{L(\mathbb{G}_n)}(\lambda)=\lambda^{2n}+ \cdots + a_2\lambda^{2}+a_1\lambda$,
    the Kirchhoff index is computed from equation (\ref{Kf_poly}) as
                        \[
                        K\!f(\mathbb{G}_{n})=2n\left|\frac{a_2}{a_1}\right|
                        \]
    for $1\leq n \leq 30$, and an unexpected result is observed in the plotted data.
                        \begin{figure}[h]
                        \centering
                        \includegraphics[width=8cm]{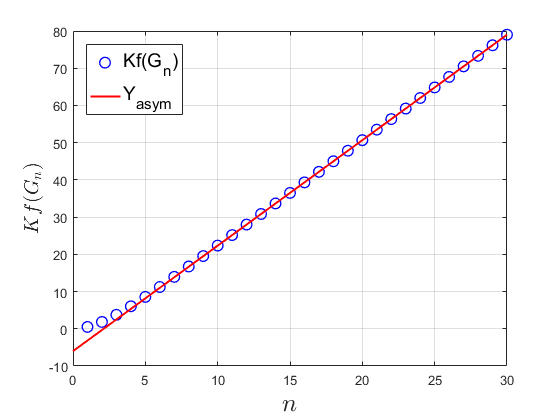}
                        \caption{\small Plot of $K\!f(\mathbb{G}_n)$ and $Y_{asym}$ for $1\leq n \leq 30$}\label{Kirchhoff_Gn_plot}
                        \end{figure}
    As shown in Figure \ref{Kirchhoff_Gn_plot}, the values (blue circles) exhibit an asymptotically linear increase with respect to $n$.
    The following theorem provides an explicit formula for $K\!f(\mathbb{G}_{n})$ as a function of $n$, which establishes that $K\!f(\mathbb{G}_{n})$ asymptotically approaches the line
                \[
                        Y_{asym}=\frac{17}{6}\,X-6=2.8\dot{3}\,X-6
                \]
    as $n \rightarrow \infty$ (see the red line in Figure \ref{Kirchhoff_Gn_plot}).
    \begin{theorem}\label{thm_Gn}
                For the sequence of graphs $\mathbb{G}_{n}$, $K\!f(\mathbb{G}_{n})$ is expressed as
                \[
                        K\!f(\mathbb{G}_{n})=\frac{17}{6}n-6+\frac{2n+9}{3\cdot 2^{n-1}}.
                \]
    \end{theorem}
    \begin{proof}
    It is observed that
                \[
                                (A_n+B_n)/4 =
                                    \setlength{\arraycolsep}{4pt}
                                        \begin{bmatrix*}[r]
                                        1 & -1 & 0 &&\cdots& 0~~ \\
                                        -1 & 3 & -2 &&&\vdots~~\\
                                        0& -2 & 6 & \ddots && \\
                                        && \ddots & \ddots & -2^{n-3} & 0~~\\
                                        \vdots&&& -2^{n-3} & 3\cdot 2^{n-3} & -2^{n-2} \\
                                        0&\cdots&\textcolor{white}{\vdots}&0& -2^{n-2} & 2^{n-2}
                                        \end{bmatrix*}
                \]
    is the Laplacian matrix of the weighted path graph (denoted by $\mathbb{P}_n$) depicted in Figure \ref{img_weighted_path}.
                \begin{figure}[h]
                        \centering
                        \includegraphics[width=12cm]{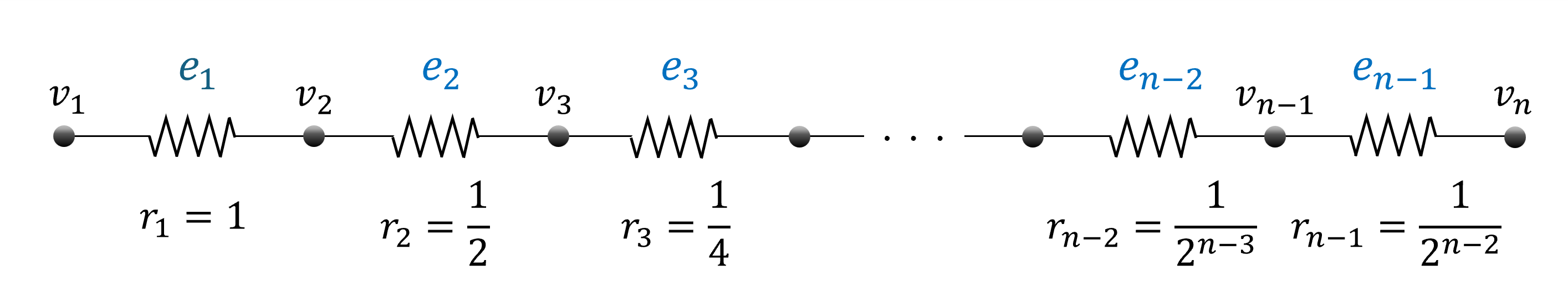}
                        \caption{\small Weighted path graph $\mathbb{P}_n$, where $L(\mathbb{P}_n)=(A_n+B_n)/4$}
                        \label{img_weighted_path}
                \end{figure}
    Assume that $0=\lambda_1\leq \lambda_2\leq \cdots \leq \lambda_n$ denote the eigenvalues of $(A_n+B_n)/4$ and that $\mu_1\leq \mu_2\leq \cdots \leq \mu_n$ denote those of $A_n-B_n$.
    Since
    $\psi_{L(\mathbb{G}_n)}(\lambda)=\psi_{A_n+B_n}(\lambda)\,\psi_{A_n-B_n}(\lambda)$
    and eigenvalues of $A_n+B_n$ are $4\lambda_i$ for $1\leq i \leq n$,
    it follows from equation (\ref{Kf_eigen}) that:
                \begin{align*}
                        K\!f(\mathbb{G}_{n})
                        &=2n\bigg(\sum_{i=2}^{n}\frac{1}{4\lambda_i}+\sum_{j=1}^{n}\frac{1}{\mu_j}\bigg)\\
                        &=\frac{1}{2}\bigg(n\sum_{i=2}^{n}\frac{1}{\lambda_i}\bigg)+2n\sum_{j=1}^{n}\frac{1}{\mu_j}\\
                        &=\frac{1}{2} K\!f(\mathbb{P}_{n}) + 2n\bigg(\frac{1}{4}+\sum_{j=2}^{n}\frac{1}{3\cdot 2^j} \bigg).
                \end{align*}
    Now, since $\mathbb{P}_{n}$ is a tree, the equivalent resistance between any two vertices $u$ and $v$ is the sum of resistances $r_k$ along the unique path connecting $u$ and $v$.
    In other words, for the edge set $E(\mathbb{T}_{n})$ of a tree $\mathbb{T}_{n}$,
                \[
                        K\!f(\mathbb{T}_{n})=\sum_{e_k\in E(\mathbb{T}_{n})}r_k\cdot S(e_k),
                \]
    where $S(e_k)$ is the number of subtrees of $\mathbb{T}_{n}$ containing an edge $e_k$.
    Therefore,
                \begin{align*}
                        K\!f(\mathbb{P}_{n})&=\sum_{k=1}^{n-1}r_k\cdot k(n-k)
                        =n\sum_{k=1}^{n-1}\frac{k}{2^{k-1}}-\sum_{k=1}^{n-1}\frac{k^2}{2^{k-1}}.
                \end{align*}
    By differentiating both sides of $\sum_{k=0}^{n-1}x^k=\frac{1-x^n}{1-x}$ with respect to $x$ and substituting $x=1/2$,
                \[
                \sum_{k=1}^{n-1}\frac{k}{2^{k-1}}=4-\frac{n+1}{2^{n-2}}
                \]
    is obtained. Similarly,
                \[
                \sum_{k=1}^{n-1}\frac{k^2}{2^{k-1}}=12-\frac{n^2+2n+3}{2^{n-2}}
                \]
    is obtained by differentiating $x\times \sum_{k=0}^{n-1}kx^{k-1}$ with respect to $x$ and evaluating at $x=1/2$.
    Hence,
                \begin{align*}
                        K\!f(\mathbb{P}_{n})
                        &=n\sum_{k=1}^{n-1}\frac{k}{2^{k-1}}-\sum_{k=1}^{n-1}\frac{k^2}{2^{k-1}}\\
                        &=n\left(4-\frac{n+1}{2^{n-2}}\right)-\bigg(12-\frac{n^2+2n+3}{2^{n-2}}\bigg)\\
                        &=4n-12+\frac{n+3}{2^{n-2}}.
                \end{align*}
    Finally,
                \begin{align*}
                        K\!f(\mathbb{G}_{n})
                        &=\frac{1}{2} K\!f(\mathbb{P}_{n}) + 2n\bigg(\frac{1}{4}+\sum_{j=2}^{n}\frac{1}{3\cdot 2^j} \bigg)\\
                        &=\frac{1}{2}\left(4n-12+\frac{n+3}{2^{n-2}}\right)
                                +2n\left(\frac{1}{4}+\frac{1}{6}\Big(1-\frac{1}{2^{n-1}}\Big) \right)\\
                        &=\frac{17}{6}n-6+\frac{2n+9}{3\cdot 2^{n-1}}.
                \end{align*}
    \end{proof}
    In the next corollary, when the sequence of graphs $\mathbb{G}_{n}$ is unweighted, $K\!f(\mathbb{G}_{n})$  can be computed more easily using the method described in the next section.
    \begin{corollary}\label{coro_Gn_unit}
    If all resistances of $\mathbb{G}_{n}$ are set to unity,
    $K\!f(\mathbb{G}_{n})$ is a cubic polynomial in $n$, given by:
                \[
                        K\!f(\mathbb{G}_{n})=
                        \begingroup
                        \renewcommand*{\arraystretch}{1.5}
                        \begin{cases}
                        \frac{1}{6}(n^3+3n^2+n) & ~n\geq 2\\[1.5ex]
                        \frac{1}{2}  &~ n=1
                        \end{cases}.
                        \endgroup
                \]
    \end{corollary}

\section{Kirchhoff index of a 4-regular graph}
      The graph $\mathbb{G}_n$ considered in this study is non-simple and all the vertices except $v_{1}$ and $v_{1'}$ are of degree 4.
      In the process of generating $\mathbb{G}_{n}$ from $\mathbb{G}_{n-1}$, instead of connecting the two multiple edges $(v_n, v_n')$, we form four edges $(v_n, v_1), (v_n, v_1'), (v_n', v_1)$, and $(v_n', v_1')$, thereby producing a 4-regular graph, as shown in Figure $\ref{img_4regular}$.
      A 4-regular graph constructed from $\mathbb{G}_{n}$ in this manner always has an even number of vertices.
      Let $\bbGamma_{2m}$ denote such a 4-regular graph with $2m$ vertices.
                \begin{figure}[h]
                        \centering
                        \includegraphics[width=12cm]{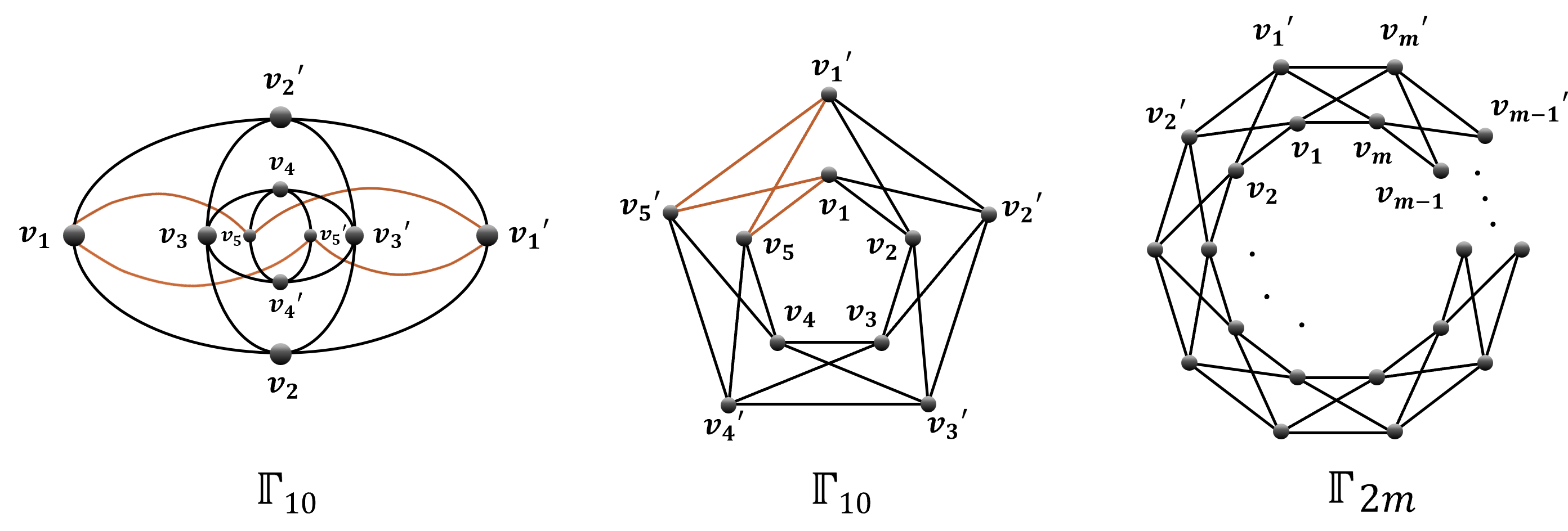}
                        \caption{4-regular graph \bbGammaT$_{10}$ and \bbGammaT$_{2m}$}
                        \label{img_4regular}
                \end{figure}
      It is found that $\bbGamma_{6}$ is the octahedron graph and $\bbGamma_{8}$ is the complete bipartite graph $\mathbb{K}_{4,4}$.
      The formula for $K\!f(\bbGamma_{2m})$ is presented in the next theorem.
      \begin{theorem}\label{thm_4reg}
      $K\!f(\bbGamma_{2m})$ of the 4-regular graph $\bbGamma_{2m}$ is give by
                \[
                        K\!f(\bbGamma_{2m})=\frac{1}{12}(m^3+6m^2-m)\qquad (m\geq 3),
                \]
      or
                \[
                        K\!f(\bbGamma_{n})=\frac{1}{96}(n^3+12n^2-4n)\qquad (n\geq 6, even).
                \]
      \end{theorem}
      \begin{proof}
      Since
                \[
                L(\bbGamma_{2m})=
                    \begingroup
                    \renewcommand*{\arraystretch}{1.1}
                            \begin{bmatrix*}[r]
                                        4I_2 & -J_2 &&&-J_2 \\
                                        -J_2 & 4I_2 & -J_2 &\textcolor{white}{\ddots}&\\
                                        & -J_2 & 4I_2 & \ddots & \\
                                        && \ddots & \ddots & -J_2  \\
                                        -J_2&\textcolor{white}{\ddots}&& -J_2 & 4I_2 \\
                            \end{bmatrix*},
                    \endgroup
                \]
    in a similar way to the proof of Theorem \ref{thm_polyofnest},
                \[
                        \psi_{L(\bbGamma_{2m})}(\lambda)
                        =\psi_{A_m+B_m}(\lambda)\,\psi_{A_m-B_m}(\lambda),
                \]
    where
                \begin{align*}
                                A_m+B_m &=
                                        \begin{bmatrix*}[r]
                                        4 & -2 &  & & -2 \\
                                        -2 & 4 & -2 &&\\
                                        & -2 & 4 &\ddots& \\
                                        & &\ddots & \ddots & -2\\
                                        -2&&& -2 & 4  \\
                                        \end{bmatrix*}
                                \quad \text{and} \quad A_m-B_m =
                                        \begin{bmatrix*}[c]
                                        4 &&&&  \\
                                        & 4 &&&\\
                                        && 4 && \\
                                        &&& \ddots &\\
                                        &&&& 4 \\
                                        \end{bmatrix*}
                \end{align*}
    We can see that $(A_m+B_m)/2$ corresponds to the Laplacian matrix of unweighted cycle graph $\mathbb{C}_m$.
    Let $\{0, \lambda_2, \lambda_3, \cdots , \lambda_m\}$ be the eigenvalues of $(A_m+B_m)/2$ and $\{\mu_1, \mu_2, \cdots , \mu_n\}$ be those of $A_m-B_m$.
    It is known that $ K\!f(\mathbb{C}_m)$ of $\mathbb{C}_m$ is expressed as \cite{lukovits1999resistance}
                \begin{equation} \label{eqn_Cn}
                        K\!f(C_m)=\frac{1}{12}(m^3-m).
                \end{equation}
    Therefore,
                \begin{align*}
                            L(\bbGamma_{2m})
                            &=2m\bigg(\sum_{i=2}^{n}\frac{1}{2\lambda_i}+\sum_{j=1}^{n}\frac{1}{\mu_j}\bigg)\\
                        &=\bigg(m\sum_{i=2}^{n}\frac{1}{\lambda_i}\bigg)+2m\sum_{j=1}^{n}\frac{1}{\mu_j}\\
                        &=K\!f(\mathbb{C}_{m}) + 2m\cdot \frac{1}{4}\cdot m\\
                        &=\frac{1}{12}(m^3+6m^2-m).
                \end{align*}
    \end{proof}

\section{Discussion}
    We examined the variation of the Kirchhoff index of a specific sequence of nested geometric graphs $\mathbb{G}_n$ and found that $K\!f(\mathbb{G}_n)$ grows almost linearly with $n$.
    Furthermore, we derived the exact solution for $K\!f(\mathbb{G}_n)$ and, subsequently, determined its asymptotic formula $Y_{asym}$.
    Among all simple connected unweighted graphs of order $n$, the path graph $\mathbb{P}_{n}$ achieves the maximum Kirchhoff index, given by
                        \[  K\!f(\mathbb{P}_n)=\frac{1}{6} (n^3-n), \]
    whereas the complete graph attains the minimum value (see \cite{lukovits1999resistance, rc1976ntringer})
                        \[ K\!f(\mathbb{K}_n)=n-1. \]
    $K\!f(\mathbb{C}_n)=1/12(n^3-n)$ for cycle graph $\mathbb{C}_n$ is introduced in equation (\ref{eqn_Cn}).
    While the Kirchhoff indices of $P_n$ and $C_n$ exhibit a cubic dependence on $n$, the complete graph $K_n$ shows a linear dependence on $n$.
                        \begin{figure}[h]
                        \centering
                        \includegraphics[width=8cm]{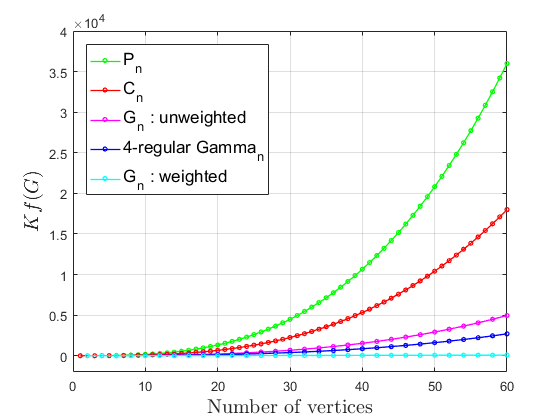}
                        \caption{\small Variation of $K\!f(\mathbb{G})$ with respect to number of vertices for five types of graphs}\label{Kirchhoff_compare}
                        \end{figure}
    Figure \ref{Kirchhoff_compare} compares the $K\!f(\mathbb{G})$ values for five classes of graphs:
    \begin{quote}
            path graph $\mathbb{P}_n$,\\
            cycle graph $\mathbb{C}_n$,\\
            $\mathbb{G}_n$ with unit resistors in Corollary \ref{coro_Gn_unit},\\
            4-regular graph $\bbGamma_{2m}$ in Theorem \ref{thm_4reg},\\
            and the original $\mathbb{G}_n$ in Theorem \ref{thm_Gn}.
    \end{quote}
    All Kirchhoff indices shown in the plot are calculated as functions of the number of vertices.

    The sequence of graphs analyzed in this study exhibits a self-repeating pattern.
    This property suggests potential applicability to geometries, such as well-known fractal structures characterized by self-similarity, which are observed across multiple scales.


\begin{thebibliography}{10}
\bibitem{klein1993resistance}
D.J. Klein, M.~Randić,
\newblock Resistance distance,
\newblock {\em J. Math. Chem.} 12 (1993) 81--95.
\bibitem{bonchev1994molecular}
D. Bonchev, A.T. Balaban, X. Liu, D.~J. Klein,
\newblock Molecular cyclicity and centricity of polycyclic graphs. I. cyclicity
  based on resistance distances or reciprocal distances,
\newblock {\em Int. J. Quantum Chem.}, 50(1) (1994) 1--20.
\bibitem{seshu1961linear}
S. Seshu, M.B. Reed,
\newblock {\em Linear graphs and electrical networks},
\newblock Addison-Wesley Publishing Company, 1961.
\bibitem{bapat2004resistance}
R.B. Bapat,
\newblock Resistance matrix of a weighted graph,
\newblock {\em MATCH Commun. Math. Comput. Chem.} 50 (2004) 73--82.
\bibitem{bapat2010graphs}
R.B. Bapat,
\newblock {\em Graphs and matrices},
\newblock Springer, 2010.
\bibitem{kelathaya2023generalized}
U. Kelathaya, R.B. Bapat, M.P. Karantha,
\newblock Generalized inverses in graph theory,
\newblock {\em AKCE Int. J. Graphs Comb.} 20(2) (2023) 108--114.
\bibitem{gutman1996quasi}
I. Gutman, B. Mohar,
\newblock The quasi-wiener and the kirchhoff indices coincide,
\newblock {\em J. Chem. Inf. Comput. Sci.} 36 (1996) 982--985.
\bibitem{bapat2003simple}
R.B. Bapat, I. Gutman, W. Xiao,
\newblock A simple method for computing resistance distance,
\newblock {\em Z. Naturforsch. A.} 58 (2003) 494--498.
\bibitem{kagan2015equivalent}
M. Kagan,
\newblock On equivalent resistance of electrical circuits,
\newblock {\em Am. J. Phys.} 83 (2015) 53--63.
\bibitem{gao2012kirchhoff}
X. Gao, Y. Luo, W. Liu,
\newblock Kirchhoff index in line, subdivision and total graphs of a regular graph,
\newblock {\em Discrete Appl. Math.} 160(4) (2012) 560--565.
\bibitem{wang2013laplacian}
W. Wang, D. Yang, Y. Luo,
\newblock The laplacian polynomial and kirchhoff index of graphs derived from regular graphs,
\newblock {\em Discrete Appl. Math.} 161(18) (2013) 3063--3071.
\bibitem{yang2008kirchhoff}
Y. Yang, H. Zhang,
\newblock Kirchhoff index of linear hexagonal chains,
\newblock {\em Int. J. Quantum Chem.} 108(3) (2008) 503--512.
\bibitem{sardar2020computation}
M.S. Sardar, X.F. Pan, S.A. Xu,
\newblock Computation of resistance distance and kirchhoff index of the two classes of silicate networks,
\newblock {\em Appl. Math. Comput.} 381 (2020) 125283.
\bibitem{liu2024extremal}
H. Liu, L. You,
\newblock Extremal kirchhoff index in polycyclic chains,
\newblock {\em Discrete Appl. Math.} 348 (2024) 292--300.
\bibitem{zhu2019normalized}
Z. Zhu, J.B. Liu,
\newblock The normalized laplacian, degree-kirchhoff index and the spanning
  tree numbers of generalized phenylenes,
\newblock {\em Discrete Appl. Math.} 254(15) (2019) 256--267.
\bibitem{sun2024resistance}
Wensheng Sun, Muhammad~Shoaib Sardar, Yujun Yang, and Shou-Jun Xu.
\newblock On the resistance distance and kirchhoff index of k n-chain (ring)
  network.
\newblock {\em Circuits, Systems, and Signal Processing}, 43(8):4728--4749, 2024.
\bibitem{horn2012matrix}
R.A. Horn, C.R. Johnson,
\newblock {\em Matrix analysis},
\newblock Cambridge university press, 2012.
\bibitem{muir2003treatise}
T. Muir, W.~H. Metzler,
\newblock {\em A Treatise on the Theory of Determinants},
\newblock Dover Publication, Inc., New York, 1960.
\bibitem{yang1985graph}
Y. Yang, T. Yu,
\newblock Graph theory of viscoelasticities for polymers with starshaped,
  multiple-ring and cyclic multiple-ring molecules,
\newblock {\em Macromol. Chem. Phys.} 186(3) (1985) 609--631.
\bibitem{Matlab}
{The MathWorks Inc.}
\newblock Matlab version: 9.14.0 (r2023a), 2023.
\bibitem{lukovits1999resistance}
I. Lukovits, S. Nikoli{\'c}, N. Trinajsti{\'c},
\newblock Resistance distance in regular graphs,
\newblock {\em Int. J. Quantum Chem.} 71(3) (1999) 217--225.
\bibitem{rc1976ntringer}
R.C. Entringer, D.E. Jackson, D.A. Snyder,
\newblock Distance in graphs,
\newblock {\em Czech. Math. J.} 26(2) (1976) 283--296.

\end{thebibliography}

\end{document}